\documentclass[11pt,a4paper]{article}
\usepackage[latin5]{inputenc}
\usepackage{amsmath}
\usepackage{amsfonts}
\usepackage{amssymb}
\usepackage{graphicx}
\usepackage{amsmath}
\usepackage{amssymb,amsmath,amsthm,amscd,latexsym}
\usepackage{amsfonts}
\newtheorem{theorem}{Theorem}[section]

\theoremstyle{definition}
\newtheorem{definition}[theorem]{Definition}
\newtheorem{example}[theorem]{Example}
\theoremstyle{remark}
\newtheorem{remark}[theorem]{Remark}
\newcommand{\te}{\theta}
\date{ }
\author{Fatmanur Gürsoy, Elif Segah Oztas and Irfan Siap\\
	Department of Mathematics,\\ Yildiz Technical University, Istanbul, TURKEY\\
	fatmanur@yildiz.edu.tr, elifsegahoztas@gmail.com, irfan.siap@gmail.com}
\title{Reversible DNA codes over $F_{16}+uF_{16}+vF_{16}+uvF_{16}$}
\begin{document}
	\maketitle
	\begin{abstract}
		In this paper we study the structure of specific linear codes called DNA codes. The first attempts on studying such codes have been proposed over four element rings which are naturally matched with DNA four letters. Later, double (pair) DNA strings or in general $k$-DNA strings called $k$-mers have been matched with some special rings and codes over such rings with specific properties are studied. However, these matchings in general are not straightforward and because of the fact that the reverse of the codewords ($k$-mers) need to exist in the code, the matching problem is difficult and it is referred to as the reversibility problem. Here, $8$-mers (DNA 8-bases) are matched with the ring elements of  $R_{16}=F_{16}+uF_{16}+vF_{16}+uvF_{16}.$  Furthermore, cyclic codes over the ring $R_{16}$ where the multiplication is taken to be noncommutative with respect to the automorphism $\te$ are studied. The preference on the skewness is shown to be very useful and practical especially since this serves as a direct solution to the reversibility problem.	
	\end{abstract}
	\section{Introduction}
	
	The interest on DNA computing is initiated  by Leonard Adleman \cite{adleman94}. Adleman solved the famous salesman problem (an NP-hard problem) in a test tube with DNA strings. Adleman considered  and made use of the well-known Watson-Crick complement (WCC) which is a relation between a DNA string and its reversible complement.
	
	DNA sequences consist of four bases (nucleotides) that are (A) Adenine, (G) Guanine, (T) Thymine and (C) Cytosine and referred as DNA letters. DNA has two strands and they are related  by the rule called Watson-Crick complement (WCC) which says that a DNA string is attached to its reversible complement string and forms a helix. Briefly the WCC of A is T and vice versa and the WCC of G is C and vice versa.

	In \cite{tehergf4}, reversible complement DNA codes were generated with DNA single bases by using  additive codes over $F_4$ which presents a matching between the four elements of the field $F_4$ and four DNA letters. Later, the first attempt to match multiple DNA letters (strings) with $F_{16}$ has been presented in \cite{ise}, where  the elements of $F_{16}$ are matched with DNA double bases. In \cite{ise}, the reversibility problem has been resolved  by introducing a $4$-power table that is a map between DNA double basis and elements of $F_{16}$. Also in \cite{agul,ise,yildizsiap},  DNA double bases are further considered. The motivation for considering strings of DNA's relies on the fact that there are identified strings (proteins) that play important role in DNA. And further there are currently many studies trying to identify the role of specific $k$-bases ($k$-mers) especially studies that focus on binding sides of DNA and trying to determine the most frequent $8$-mers \cite{alper}.
	
	Firstly, the ring on  which linear codes are defined in this paper is  a $16^4$ element commutative non chain ring
	$$ R_{16}=F_{16}[u,v]/\langle u^2-u, v^2-v\rangle=F_{16}+uF_{16}+vF_{16}+uvF_{16}$$
	
	where $F_{16}$ is a $16$-element field.
	
	Clearly, $ R_{16}=F_{16}+uF_{16}+vF_{16}+uvF_{16}=\{a+ub+vc+uvd |a,b,c,d\in F_{16}, u^2=u, \ v^2=v\}.$
	
	In order to clarify the reversibility problem we present a concrete  example.  Let $(\alpha_1,\alpha_2,\alpha_3)$ be a codeword corresponding to AATTGGCCTTTT (a $12$-string)
	where $\alpha_1 \rightarrow $AATT, $\alpha_2 \rightarrow $GGCC, $\alpha_3 \rightarrow  $TTTT and $\alpha_1 ,\alpha_2 ,\alpha_3 \in R_{16}.$  The reverse of $(\alpha_1,\alpha_2,\alpha_3)$ is $(\alpha_3,\alpha_2,\alpha_1)$, and $(\alpha_3,\alpha_2,\alpha_1)$
	corresponds to TTTTGGCCAATT. However, TTTTGGCCAATT is not the reverse of AATTGGCCTTTT. Indeed, the reverse of AATTGGCCTTTT is TTTTCCGGTTAA.

	Note that, in this study, an element of $F_{16}+uF_{16}+vF_{16}+uvF_{16}$ corresponds to a DNA 8-bases. For example $\alpha+u+\alpha^3v+\alpha^2 uv$ corresponds to GACATAAT.
	
	Some studies on multiple DNA letter matchings which naturally leads to the reversibility problem can be found in \cite{tehergf4,agul,ise,siaptaherr,yildizsiap}, where the  researchers have aimed  to solve the reversibility problem on DNA codes and generated reversible and reversible complement DNA codes.
	
	In this study, first we explore skew cyclic codes over $F_{16}+uF_{16}+vF_{16}+uvF_{16}$. Then, we present a solution to the reversibility problem for DNA 8-bases and obtain reversible DNA codes.   We accomplish this task by considering skew cyclic codes over a skew polynomial ring  defined via an automorphism and by choosing  special factors of $x^n-1.$

	\section{Preliminaries and definitions}
	
	In this section we  give some basic properties and definitions of skew cyclic codes and reversible DNA codes.
	\begin{definition}\cite{Jacobson}
		Let $R$ be a commutative ring with identity and $\theta$ be an automorphism over $R$. The set of polynomials $R[x;\theta]=\{a_0+a_1x+...+a_{n-1}x^{n-1}|a_i\in R, n\in \mathbb{N}\}$ is called the skew polynomial ring over $R$ where addition is the usual addition of polynomials and the multiplication is defined by $xa=\theta(a)x$ ($a\in R$) and extended to polynomial multiplication naturally.
	\end{definition}
	Skew cyclic codes were originally introduced by Boucher et al. in \cite{skew cyclic} by using skew polynomial rings over the finite field $F_q$. A skew cyclic code is defined to be a linear code (an $R$-submodule of $R^n$) $C$ of length $n$ over $R$ which further satisfies the property that $(\te(c_{n-1}),\te(c_0),...,\te(c_{n-2}))\in C$, for all $(c_0,c_1,...,c_{n-1})\in C$. In polynomial representation, a skew cyclic code $C$ of length $n$ over $F_q$ corresponds to a left ideal of the quotient ring $F_q[x;\theta]/(x^n-1)$, if the order of $\theta$, say $m$, divides $n$ \cite{skew cyclic}. If $m$ does not divide $n$ then $F_q[x;\theta]/(x^n-1)$ is not a ring anymore. In this case the skew cyclic code $C$ can be  considered as left $F_q[x;\te]$-submodule of  $F_q[x;\te]/(x^n-1)$ \cite{I.siap}. In both cases $C$ is generated by a monic polynomial $g(x)$ which is a right divisor of $x^n-1$ in $F_q[x;\te]$ and denoted by $C=( g(x))$.
	Recently, definition of skew cyclic codes over finite fields has been extended to rings. For instance;  skew cyclic codes are defined over finite chain rings in \cite{jitman}, over Galois rings in \cite{galois}, over $F_q+vF_q$ in \cite{fq+vfq}.
	
	
	\begin{definition}
		Let $C$ be a code of length $n$ over $F_q$. If $c^r=(c_{n-1},c_{n-2},\ldots,c_1,c_0)\in C$ for all $c=(c_0,c_1,\ldots,c_{n-1})\in C$, then $C$ is called a reversible code.
	\end{definition}

	\section{Reversible DNA codes over $R_{16}$}

	In this study, we  define skew cyclic codes over the ring $R_{16}=F_{16}+uF_{16}+vF_{16}+uvF_{16}$. $R_{16}$ is a commutative non-chain ring where $u^2=u$, $v^2=v$ and $uv=vu$. By Chinese remainder theorem we can decompose $R_{16}$ as follows:
	$R_{16}=uvF_{16}\oplus(v+uv)F_{16}\oplus(u+uv)F_{16}\oplus(1+u+v+uv)F_{16}$. We define a Gray map;
	\begin{equation}
	\begin{split}
	\phi:R_{16} &\rightarrow F_{16}^4  \\
	a+ub+vc+uvd &\rightarrow (a+b+c+d, \ a+c, \ a+b, \ a).
	\end{split}
	\end{equation}
	In \cite{yao}, skew cyclic codes over the ring  $R_q$, where $q=p^m$ and $p$ is an odd prime, are defined with respect to the automorphism $\te'(a+ub+vc+uvd)=a^p+uc^p+vb^p+uvd^p$.
	Here we study skew cyclic codes with  $q=16$ and introduce a new automorphism over $R_{16}$;
	
	\begin{equation}\label{auto}
	\begin{split}
	\te:R_{16} &\rightarrow R_{16}  \\
	a+ub+vc+uvd &\rightarrow a^4+(1+u)b^4+(1+v)c^4+(1+u)(1+v)d^4\\
	&=(a+b+c+d)^4+u(b+d)^4+v(c+d)^4+uvd^4.
	\end{split}
	\end{equation}
	Naturally, $R_{16}[x;\te]$ is a skew polynomial ring and a skew cyclic code $C$ of length $n$ over $R_{16}$ corresponds to a left $R_{16}$-submodule of $R_{16}[x;\te]/(x^n-1)$. Because of the complexity of classifying all left submodules of $R_{16}[x;\te]/(x^n-1)$ we   make use of the Gray image of $C$ by defining the following subspaces over $F_{16}$.
	\begin{align*}
	C_1&=\{\textbf{a}+\textbf{b}+\textbf{c}+\textbf{d}\in F_{16}^n\rvert \textbf{a}+u\textbf{b}+ v \textbf{c} +uv \textbf{d} \in C  \} \\
	C_2&=\{\textbf{a}+\textbf{c}\in F_{16}^n\rvert \textbf{a}+u\textbf{b}+ v \textbf{c} +uv \textbf{d} \in C, \text{ for some }  \textbf{b},\textbf{d} \in F_{16}^n \}\\
	C_3&=\{\textbf{a}+\textbf{b}\in F_{16}^n\rvert \textbf{a}+u\textbf{b}+ v \textbf{c} +uv \textbf{d} \in C, \text{ for some }  \textbf{c},\textbf{d} \in F_{16}^n \}\\
	C_4&=\{\textbf{a}\in F_{16}^n\rvert \textbf{a}+u\textbf{b}+ v \textbf{c} +uv \textbf{d} \in C, \text{ for some }  \textbf{b},\textbf{c},\textbf{d} \in F_{16}^n \}
	\end{align*}
	
	It can be easily seen that $C=uvC_1\oplus(v+uv)C_2\oplus(u+uv)C_3\oplus(1+u+v+uv)C_4$ and each $C_i$ is a linear code of length $n$ over $F_{16}$ for $i\in \{1,2,3,4\}$.
	
	In \cite{ise}, Table 1 gives a mapping $\tau$ between  elements of $F_{16}$ and DNA pairs in such a way that each element of $F_{16}$ and its $4$th power are mapped to DNA pairs, which are reverses of each other.  For example, in Table 1 $\tau(\alpha)=$GC while $\tau(\alpha^4)=$CG.  This map can be naturally extended to a map $\tau_2$ from $F_{16}^4$ to DNA 8-bases as follows;   $\tau_2(a,b,c,d)=(\tau(a),\tau(b),\tau(c),\tau(d))$ where $a,b,c,d\in F_{16}$.
	Throughout the paper we will use $\alpha$ as in Table 1.
	
	To make the connection between skew cyclic codes over $R_{16}$ and DNA codes we  define a map   $\varphi=\tau_2 \circ \phi$ and so for any $a+ub+vc+uvd\in R_{16}$; $\varphi(a+ub+vc+uvd)=\tau_2(\phi(a+ub+vc+uvd))=(\tau(a+b+c+d), \ \tau(a+c), \ \tau(a+b), \ \tau(a))$. In this way for any $\beta=a+ub+vc+uvd\in R_{16}$;  $\varphi(\beta)$ and $\varphi(\te(\beta))$ are DNA reverses of each other, since $\varphi(\te(\beta))=\tau_2(a^4,(a+b)^4,(a+c)^4,(a+b+c+d)^4)$. This map can naturally be extended to $n$-tuples coordinatewise. For any $c=(c_0, c_1, \ldots, c_{n-1})\in R_{16}^n $ we have $c'=(\te (c_{n-1}),  \ldots, \te (c_1), \te (c_0))\in R_{16}^n $,  that is, $\varphi(c)^r=\varphi(c')$, or equivalently, $(\varphi(c_0), \varphi(c_1), \ldots, \varphi( c_{n-1}))^r=(\varphi(\te (c_{n-1})), \ldots, \varphi(\te(c_1)),  \varphi( \te(c_{0})))$. 
	
	\begin{table}
		\caption{ The $\tau$ mapping between DNA pairs and $F_{16}$ \cite{ise}}\label{dna table}
		\begin{tabular}{lcl}
			\hline
			$F_{16}$(multiplicative)&  $F_{16}$(additive)&Double DNA pair \\
			\hline\hline
			0&0&AA\\
			$\alpha^0$ &1&TT \\
			$\alpha^1$ &$\alpha$&AT\\
			$\alpha^2$ & $\alpha^2$&GC\\
			$\alpha^3$ & $\alpha^3$&AG \\
			$\alpha^4$ & $1+\alpha$&TA  \\
			$\alpha^5$ & $ \alpha+\alpha^2$ &CC \\
			$\alpha^6$ & $\alpha^2 +\alpha^3$&AC \\
			$\alpha^7$ & $1+\alpha +\alpha^3$&GT  \\
			$\alpha^8$ & $1 +\alpha^2$&CG \\
			$\alpha^9$ & $\alpha +\alpha^3$&CA \\
			$\alpha^{10}$ & $1+\alpha +\alpha^2$&GG \\
			$\alpha^{11}$ & $\alpha +\alpha^2+\alpha^3$&CT \\
			$\alpha^{12}$ & $1+\alpha +\alpha^2+\alpha^3$&GA \\
			$\alpha^{13}$ & $1 +\alpha^2+\alpha^3$&TG \\
			$\alpha^{14}$ & $1+\alpha^3$&TC \\
		\end{tabular}
	\end{table}
	\begin{example}
		Let $\beta=\alpha+u+\alpha^3v+\alpha^2 uv\in R_{16}$. Then,  $\varphi(\beta)=(\tau(\alpha^{12}), \tau(\alpha^9), \tau(\alpha^4),$ $ \tau(\alpha)) = $(GA,CA,TA,AT). And $\varphi	(\te(\beta))=\tau_2(\alpha^4, \alpha, \alpha^6, \alpha^3)=$ (TA,AT,AC,AG)
		Therefore,  $\varphi(\beta)^r=\varphi(\te(\beta))$.
	\end{example}
	
	\begin{definition}
		Let $C\subseteq R_{16}^n$. If $\varphi(c)^r\in \varphi(C)$  for all $c\in C$, then $C$ or equivalently $\varphi(C)$ is called a reversible DNA code.
	\end{definition}

	\begin{definition}\cite{alger}
		Let $f(x)=a_0+a_1x+\ldots +a_tx^t$ be a polynomial of degree $t$ over  $ R_{16}$.  $f(x)$ is said to be a palindromic polynomial if $a_i=a_{t-i}$  for all $i\in  \{0,1,\ldots,t \}$.
		And $f(x)$ is said to be a $\theta$-palindromic polynomial if $a_i=\theta(a_{t-i})$  for all $i\in \{0,1,\ldots,t\}$.
	\end{definition}

	Let $C$ be a skew cyclic code of length $n$ over $F_q$ with respect to an automorphism $\te'$. If the order of $\te'$ and $n$ are relatively prime then $C$ is a cyclic code over $F_q$ \cite{I.siap}. Similarly, any skew cyclic code of odd length over $R_{16}$ with respect to  $\te$ is a cyclic code, since the order of $\te$ is $2$. For this reason we restrict the length $n$ to even numbers only.
	
	In \cite{ise}, reversible DNA codes were obtained by indirect methods such as using lifted polynomials. However, here by considering skew cyclic codes with special generators
	we are able to construct  reversible DNA codes directly. The skewness property not only provides a richer source for DNA codes but also proves to be more practical.
	
	\begin{theorem}\label{teo1}
		Let $C=( g(x))$ be a skew cyclic code of length $n$  over $R_{16}$ where  $g(x)$ is a right divisor of $x^n-1$ in $R_{16}[x;\theta]$ and $deg(g(x))$ is odd. If  $g(x)$ is a $\theta$-palindromic polynomial then $\varphi(C)$ is a reversible DNA code.
	\end{theorem}
	
	\begin{proof}
		Let $g(x)$ be a $\theta$-palindromic polynomial. Recall that $\varphi$ gives the correspondence of codewords in DNA form. Reverses of each DNA codeword $\varphi(c)$, for $c\in C$, are obtained by the following equation:
		\begin{equation}
		\varphi \left(  \sum_{i=0} ^{k-1}\beta_i x^i g(x)\right)^r =  \varphi  \left(  \sum_{i=0} ^{k-1} \theta(\beta_i) x^{k-1-i}g(x) \right)
		\end{equation}
		where $k=n-deg(g(x))$ and $\beta_i  \in R_{16}$ and note that, we do not distinguish between the vector representation and the polynomial representation of a codeword in $R_{16}^n$. Since $ \sum_i \theta(\beta_i) x^{k-1-i}g(x)\in C$, then $\varphi(C)$ is a  reversible DNA code. 
	\end{proof}
	\begin{example}
		$x^6-1=h(x)g(x)=(1+(\alpha^7+\alpha(u+v))x+(\alpha^7+\alpha(u+v))x^2+x^3) (1+(\alpha^7+\alpha(u+v))x+(\alpha^{13}+\alpha^4(u+v))x^2+x^3)$ in $R_{16}[x;\te]$.  Then $C=(g(x))$ is a skew cyclic code over $R_{16}$ with the parameters [6,3,4].
		Since the degree of $g(x)$ is odd and it is a $\theta$-palindromic polynomial thus  $\varphi(C)$ is a reversible DNA code.
		
	\end{example}
	\begin{theorem}\label{teo2}
		Let $C=( g(x))$ be a skew cyclic code of length $n$  over $R_{16}$ where  $g(x)$ is a right divisor of $x^n-1$ in $R_{16}[x;\theta]$ and $deg(g(x))$ is even. If  $g(x)$ is a palindromic polynomial then $\varphi(C)$ is a reversible DNA code.
	\end{theorem}
	\begin{proof}
		Let $g(x)$ be a palindromic polynomial. Recall that $\varphi$ gives the correspondence of codewords in DNA form. Reverses of each DNA codeword $\varphi(c)$, for $c\in C$, are obtained by the following equation:
		\begin{equation}
		\varphi \left(  \sum_{i=0} ^{k-1}\beta_i x^i g(x)\right)^r =  \varphi  \left(  \sum_{i=0} ^{k-1} \theta(\beta_i) x^{k-1-i}g(x) \right)
		\end{equation}
		where $k=n-deg(g(x))$ and $\beta_i \in R_{16}$ and note that, we do not distinguish between the vector representation and the polynomial representation of a codeword in $R_{16}^n$. Since $ \sum_i \theta(\beta_i) x^{k-1-i}g(x)\in C$, then $\varphi(C)$ is a  reversible DNA code.
	\end{proof}
	\begin{example}
		$x^6-1=h(x)g(x)=(1+(\alpha^3+\alpha^2(u+v))x+x^2)(1+(\alpha^3+\alpha^2(u+v))x+(\alpha^3+\alpha^2(u+v))x^3+x^4)$ in $R_{16}[x;\te]$. Then $C=(g(x))$ is a skew cyclic code over $R_{16}$ with the parameters [6,4,4].
		Since the degree of $g(x)$ is even and it is a palindromic polynomial thus  $\varphi(C)$ is a reversible DNA code.
	\end{example}
	\begin{remark}
		To illustrate the proof of Theorem \ref{teo1} and Theorem \ref{teo2}, we let $C=( g(x))$ be a skew cyclic code of length $8$  over $R_{16}$. Suppose that $g(x)$ is a $\te$-palindromic polynomial of degree $5$, then $g(x)=g_0+g_1 x+g_2x^2+\te(g_2)x^3+\te(g_1)x^4+\te(g_0)x^5$, for some $g_0,g_1,g_2\in R_{16}$. In vector representation $g(x)$ corresponds to the codeword $c=(g_0,g_1,g_2,\te(g_2),\te(g_1),\te(g_0),0,0)\in C.$ Now,
		\begin{align*}
		\varphi(c)^r&= (\varphi(g_0),\varphi(g_1),\varphi(g_2),\varphi(\te(g_2)),\varphi(\te(g_1)),\varphi(\te(g_0)),\varphi(0),\varphi(0))^r\\
		&= (\varphi(0),\varphi(0),\varphi(g_0),\varphi(g_1),\varphi(g_2),\varphi(\te(g_2)),\varphi(\te(g_1)),\varphi(\te(g_0)))\\
		&=\varphi(c'),
		\end{align*}
		where $c'= (0,0,g_0,g_1,g_2,\te(g_2),\te(g_1),\te(g_0))$. Since $c'$ is the vector representation of the polynomial $x^2g(x)$ we have $c'\in C$. Hence $g(x)$ and $x^2g(x)$ correspond to the DNA codewords that are DNA reverses of each other. Similarly, for any $\beta\in R_{16}$, corresponding DNA codewords of the polynomials  $\beta g(x)$ and $\te(\beta)x^2g(x)$ are DNA reverses of each other.
		
	\end{remark}
	\section{Conclusion}
	
	In this study, the algebraic structure of skew cyclic codes with some specific properties that lead to construction of reversible DNA codes is studied. Here,  in order to establish a matching to $8$-mers, 
	a specific ring is considered and codes over this ring are studied. It is also observed that  the skew property of the polynomial ring serves better than the case over commutative rings. Future studies over more general $k$-mers $(k>8)$ and the dual codes over these rings as DNA codes are still interesting problems to be considered.

	\textbf{Acknowledgment:}
	The authors wish to express their thanks to the anonymous referees  whose comments improved the presentation of the paper.

	\medskip
	Received xxxx 20xx; revised xxxx 20xx.
	\medskip
\end{document}